\documentclass[a4paper, 11pt]{amsart}
\title[BG inequality in positive characteristic]
{On the Bogomolov-Gieseker inequality in positive characteristic}
\date{}
\author{Naoki Koseki}

\makeatletter
 
  \@addtoreset{equation}{section}
 \makeatother

\usepackage{amsmath, amssymb, amsthm, amscd, comment, mathtools, color}
\usepackage[frame,cmtip,curve,arrow,matrix,line,graph]{xy}
\usepackage{tikz}
\usetikzlibrary{intersections, calc}

\theoremstyle{plain}
\newtheorem{thm}{Theorem}[section]
\newtheorem{prop}[thm]{Proposition}
\newtheorem{def-prop}[thm]{Definition-Proposition}
\newtheorem{lem}[thm]{Lemma}

\newtheorem*{thm*}{Theorem}

\theoremstyle{definition}
\newtheorem{defin}[thm]{Definition}

\newtheorem*{NaC}{Notation and Convention}
\newtheorem*{ACK}{Acknowledgement}

\theoremstyle{remark}
\newtheorem{rmk}[thm]{Remark}

\newtheorem{ex}[thm]{Example}

\DeclareMathOperator{\ch}{ch}

\newcommand{\dR}{\mathbf{R}}

\newcommand{\bP}{\mathbb{P}}
\newcommand{\bC}{\mathbb{C}}
\newcommand{\bR}{\mathbb{R}}
\newcommand{\bQ}{\mathbb{Q}}
\newcommand{\bZ}{\mathbb{Z}}
\newcommand{\mcA}{\mathcal{A}}

\newcommand{\mcD}{\mathcal{D}}

\newcommand{\mcF}{\mathcal{F}}

\newcommand{\mcH}{\mathcal{H}}

\newcommand{\mcO}{\mathcal{O}}

\newcommand{\mcS}{\mathcal{S}}
\newcommand{\mcT}{\mathcal{T}}

\newcommand{\mcV}{\mathcal{V}}

\DeclareMathOperator{\Hom}{Hom}

\DeclareMathOperator{\Coh}{Coh}

\DeclareMathOperator{\Ker}{Ker}

\DeclareMathOperator{\Pic}{Pic} 
\DeclareMathOperator{\NS}{NS}

\DeclareMathOperator{\Image}{Im}

\DeclareMathOperator{\cl}{cl}

\DeclareMathOperator{\BD}{BD}
\DeclareMathOperator{\cha}{char}
\DeclareMathOperator{\Bs}{Bs}
\DeclareMathOperator{\pr}{pr}

\newcommand{\Step}[1]{\smallskip\paragraph{\sc{Step #1}}}

\begin{document}
\maketitle

\begin{abstract}
We prove a version of the Bogomolov-Gieseker inequality 
on smooth projective surfaces of general type in positive characteristic, 
which is stronger than the result by Langer 
when the ranks of vector bundles are sufficiently large. 
Our inequality enables us to construct Bridgeland stability conditions 
with full support property on all smooth projective surfaces 
in positive characteristic. 
We also prove the BG type inequality 
for higher diminsional varieties. 
\end{abstract}

\setcounter{tocdepth}{1}
\tableofcontents

\section{Introduction}
In the theory of algebraic surfaces in characteristic zero, 
the Bogomolov-Gieseker (BG) inequality \cite{bog78, gie79} 
is a powerful tool, which is the inequality 
for the Chern characters 
of slope semistable vector bundles: 
\begin{equation} \label{eq:bgintro}
\Delta(E)=\ch_1(E)^2-2\ch_0(E)\ch_2(E) \geq 0. 
\end{equation}

Among others, one of the important consequences of the BG inequality 
is the existence of Bridgeland stability conditions on surfaces 
\cite{ab13, bri07, bri08}. 
The theory of Bridgeland stability conditions on surfaces 
has been applied to various classical problems in algebraic geometry 
such as birational geometry \cite{abch13, bmw14, ch14, ch15, ch16, chw17, lz16, lz18}, 
Brill-Noether problem \cite{bay18, bl17}, 
higher rank Clifford indices \cite{fl18, li19b, kos20}, 
and so on. 

In positive characteristic, it is known that 
the BG inequality (\ref{eq:bgintro}) 
does not hold in general 
(cf. \cite{muk13, ray78}). 
When a surface $S$ has Kodaira dimension $\kappa(S) \leq 1$ 
and it is not quasi-elliptic, 
Langer \cite{lan16} proves that 
the inequality (\ref{eq:bgintro}) holds. 
For surfaces of general type, 
although Langer \cite{lan04, lan15} proves 
the inequality with a modified term, 
it is still a mysterious problem 
what the best possible form of the inequality is. 
In particular, the result in \cite{lan04, lan15} 
is not enough to construct Bridgeland stability conditions 
on surfaces in positive characteristic. 

In this paper, we investigate 
the improvement of Langer's results \cite{lan04, lan15} 
for general type surfaces. 
The following is our main result in this paper: 

\begin{thm}[Theorem \ref{thm:bg1}]
\label{thm:bgintro}
Let $S$ be a smooth projective surface 
defined over an algebraically closed field of positive characteristic. 
Then there exists a constant $C_S \geq 0$, 
depending only on a birational equivalence class of $S$, 
satisfying the following condition: 
For every numerically non-trivial nef divisor $H$ on $S$ 
and $\mu_H$-semistable torsion free sheaf $E \in \Coh(S)$ on $S$, 
the inequality 
\[
\Delta(E)+C_S\ch_0(E)^2 \geq 0. 
\]
holds. Explicitly, we can take the constant $C_S$ as follows: 
\begin{enumerate} 
\item When $S$ is a minimal surface of general type, then 
\[
C_S=2+5K_S^2-\chi(\mcO_S).
\] 
\item When $\kappa(S)=1$ and $S$ is quasi-elliptic, then $C_S=2-\chi(\mcO_S)$.
\item Otherwise, $C_S=0$. 
\end{enumerate}
\end{thm}

The above theorem is strong enough to 
construct Bridgeland stability conditions on 
an arbitrary surface in positive characteristic: 

\begin{thm}[Theorem \ref{thm:stab}]
\label{thm:stabintro}
Let $S$ be a smooth projective surface defined 
over an algebraically closed field of positive characteristic. 
Then there exist Bridgeland stability conditions on $D^b(S)$ 
satisfying a full support property. 
\end{thm}

We also show a higher dimensional version of 
Theorem \ref{thm:bgintro}: 
\begin{thm}[Theorem \ref{thm:BGhigh}]
\label{thm:bghigherintro}
Let $X$ be a smooth projective variety of dimension $n \geq 2$, 
defined over an algebraically closed field of positive characteristic, 
and $H$ a very ample divisor on $X$. 
Let $H_1, \cdots, H_{n-2} \in |H|$ be general hyperplanes 
and put $S:=H_1 \cap \cdots \cap H_{n-2}$. 
Define a constant $C_{X, H}:=C_S$ as in Theorem \ref{thm:bgintro}. 
Then for every $\mu_H$-semistable torsion free sheaf $E \in \Coh(X)$, 
we have 
\[
H^{n-2}\Delta(E)+C_{X, H}\ch_0(E)^2 \geq 0. 
\]
\end{thm}

Specializing to hypersurfaces in the projective spaces, 
we have the following better result: 
\begin{thm}[Theorem \ref{thm:hyper}] 
\label{thm:hyperintro}
Let $k$ be an algebraically closed field of positive characteristic. 
Let $S^n_d \subset \bP^{n+1}_k$ be a smooth hypersurface 
of degree $d \geq 1$, dimension $n \geq 2$. 
Denote by $H$  the restriction of the hyperplane class on $\bP^{n+1}$. 
Then for every $\mu_H$-semistable torsion free sheaf $E \in \Coh(S^n_d)$, 
we have 
\[
\overline{\Delta}_H(E):=\left(H^{n-1}\ch_1(E) \right)^2-2H^n\ch_0(E)H^{n-2}\ch_2(E) \geq 0. 
\] 

In particular, if $\Pic(S^n_d)=\bZ[H]$, then we have 
the usual BG inequality $H^{n-2}\Delta(E) \geq 0$. 
\end{thm}

\subsection{Idea of proof}
One big part of the proof is 
the argument similar to the one in \cite{lan04}, 
where he proves the BG inequality in positive chararcteristic, 
with the modified term depending on $\ch_0(E)^4$. 
In addition to his arguments, 
the key fact we use is the``invariance" of the BG inequality 
under blow-ups and change of polarizations, 
discussed in \cite{blms17, lan16}. 
According to these invariance, 
we are able to proceed by induction on $\ch_0(E)$ 
in the birational equivalence class of a surface, 
rather than the whole category of surfaces. 
As a result, we obtain a quadratic inequality as in Theorem \ref{thm:bgintro}, 
which is necessary for the construction of 
Bridgeland stability conditions in Theorem \ref{thm:stabintro}. 
It would be worth mentioning that 
our proof simplifies Langer's original proof 
so that we do not need the careful analysis 
of slope semistability under the Frobenius pull-backs. 

For  higher dimension, 
to ensure the well-definedness of the constant $C_{X, H}$
 in Theorem \ref{thm:bghigherintro}, 
we use a result from the minimal model program 
for threefolds in positive characteristic. 

For Theorem \ref{thm:hyperintro}, 
besides Langer's induction argument, 
the key input is the BG inequality for surfaces $S^2_d$, 
which is proved in \cite{kos20b} 
using the theory of tilt-stability conditions on the derived categories. 

\subsection{Relation with the existing works}
Theorem \ref{thm:bgintro} was previously known 
when the Kodaira dimenstion $\kappa \leq 1$, 
except quasi-elliptic surfaces (\cite{lan16}). 
Our main contribution is the case of $\kappa=2$. 
See \cite{sun19} for the case of 
product type varieties. 
Note also that Shepherd-Barron \cite{sb91} 
obtains a similar result for rank two bundles on surfaces. 

In \cite{blms17}, the authors consider 
a variant of Langer's induction argument 
in the category of codimension two blow-ups 
of a given variety. 
The present paper is inspired by them. 

The construction of Bridgeland stability conditions 
on surfaces in positive characteristic 
has been an open problem 
since its appearance.

\subsection{Plan of the paper}
This paper is organized as follows. 
In Section \ref{sec:pre}, 
we recall the notion of slope stability, 
and the result from the minimal model program. 
In Section \ref{sec:surf}, 
we prove Theorem \ref{thm:bgintro}. 
In Section \ref{sec:high}, 
we prove Theorem \ref{thm:bghigherintro}. 
In Section \ref{sec:hyper}, we discuss 
the BG inequality for hypersurfaces in the projective spaces, 
and prove Theorem \ref{thm:hyperintro}. 
In Section \ref{sec:stab}, 
we recall the notion of Bridgeland stability conditions 
and prove Theorem \ref{thm:stabintro}.

\begin{ACK}
The author would like to thank Professor Arend Bayer
for useful discussions and comments, 
and Professor Adrian Langer 
for insightful comments. 
The author would also like to thank 
Masaru Nagaoka, Kenta Sato, 
and Professor Hiromu Tanaka 
for discussions related to MMP. 
In particular, the author learned the proof of Theorem \ref{thm:mmp} 
from Professor Hiromu Tanaka. 
The author was supported by 
ERC Consolidator grant WallCrossAG, no.~819864. 
\end{ACK}

\begin{NaC}
Throughout the paper, we work over an algebraically closed field of 
characteristic $p >0$. 
We use the following notations: 
\begin{itemize}
\item $\Coh(X)$: the category of coherent sheaves on a variety $X$. 
\item $D^b(X):=D^b(\Coh(X))$: the bounded derived category of coherent sheaves. 
\item $\ch^B=(\ch^B_0, \ch^B_1, \cdots, \ch^B_n):=e^{-B}.\ch$: 
the $B$-twisted Chern character 
for an $\bR$-divisor $B$. 
\end{itemize}
\end{NaC}

\section{Preliminaries} \label{sec:pre}
\subsection{Notations on slope stability}
Let $X$ be a smooth projective variety 
of dimension $n \geq 1$, 
defined over an algebraically closed field $k$ 
of characteristic $\cha(k)=p > 0$. 
We introduce the basic notions used in this paper: 

\begin{defin}
Let $D_1, \cdots, D_{n-1}$ be a collection of nef divisors 
such that the $1$-cycle $D_1 \dots D_{n-1}$ is numerically non-trivial. 
For a torsion free sheaf $E \in \Coh(X)$, 
we define the {\it $(D_1 \cdots D_{n-1})$-slope} of $E$ as 
\[
\mu_{D_1 \dots D_{n-1}}(E)
:=\frac{D_1 \dots D_{n-1}\ch_1(E)}{\ch_0(E)}. 
\]

We then define the notion of 
{\it $\mu_{D_1 \dots D_{n-1}}$-(semi)stability} 
(or {\it $(D_1 \dots D_{n-1})$-(semi)stability)}
in the usual way. 
When we have $H=D_1=\cdots=D_{n-1}$, 
we also call it as $\mu_H$-stability. 
\end{defin}

\begin{defin}
For a coherent sheaf $E \in \Coh(X)$, 
we define the {\it discriminant} of $E$ as follows:
\[
\Delta(E):=\ch_1(E)^2-2\ch_0(E)\ch_2(E). 
\]

For an ample divisor $H$ on $X$, we also define 
the {\it $H$-discriminant} as 
\[
\overline{\Delta}_H(E):=\left(H^{n-1}\ch_1(E) \right)^2-2H^n\ch_0(E)H^{n-2}\ch_2(E). 
\]
\end{defin}

\subsection{A result from minimal model program}
\begin{defin}
We say a morphism $f \colon X \to Z$ is a 
{\it projective contraction} if 
$X$ and $Z$ are quasi-projective varieties 
and $f$ is a projective morphism 
satisfying $f_*\mcO_X=\mcO_Z$. 
\end{defin}

We use the following result from minimal model program (MMP) 
for threefolds in positive characteristic, 
built on the works \cite{bir16, bw17, ct19, gnt19, hw19, hw20, hx15, tan20}: 

\begin{thm}
\label{thm:mmp}
Let $C$ be a smooth quasi-projective curve 
defined over an algebraically closed field of characteristic $p>0$. 
Let $f \colon X \to C$ be a projective contraction. 
Suppose that the following conditions hold: 
\begin{enumerate}
\item The variety $X$ is a $\bQ$-factorial terminal threefold, 
\item For every closed point $c \in C$, 
the fiber $X_c:=f^{-1}(c)$ is an irreducible surface, 
and the pair $(X, X_c)$ is plt. 
\end{enumerate}

Then we can run a $K_X$-MMP over $C$. 
The MMP terminates with a minimal model over $C$, 
or a Mori fiber space. 
\end{thm}
In particular, we can apply the above theorem 
when the morphism $f \colon X \to C$ is 
smooth of relative dimension two. 

\begin{proof}
When $p \geq 5$, 
the assertion directly follows from \cite[Theorem 1.2]{hw19}. 
We give a proof based on \cite{hw20}, 
which works in an arbitrary characteristic. 


\Step1(Cone theorem and Contraction theorem)
By the same argument as in \cite[Proof of Theorem 1.6]{hw20}, 
Cone theorem and Contraction theorem hold 
for the morphism $X \to C$. 
Hence if $X$ is not a minimal model over $C$, 
then there exists the contraction 
$\phi \colon X \to \overline{X}$ 
(over C) of a $K_X$-negative extremal ray. 
Note that if it is a divisorial contraction, 
then the variety $X_1:=\overline{X}$ satisfies 
the conditions (1) and (2). 

\Step2(Existence of a flip)
If $\phi$ is a flippling contraction, 
let $R$ be the curve contracted by the morphism $\phi$. 
Then $R$ is contained in an $f$-fiber $X_o:=f^{-1}(o)$ 
for some closed point $o \in C$. 
By assumption (2), the pair $(X, X_o)$ is plt and we have $X_o.R=0$. 
Hence by \cite[Proposition 4.1]{hw20}, 
there exists a flip $f_1 \colon X_1 \to C$ of $\phi$. 
Furthermore, the pair $(X_1, f_1^{-1}(o))$ is again plt. 
Hence $X_1$ satisfies the conditions (1) and (2). 

\Step3(Termination)
It remains to show the finiteness of terminal flips. 
By \cite[Corollary 3.10]{kol13}, 
the singular locus of a terminal threefold is zero dimensional. 
Hence we can check that the proof of \cite[Theorem 6.17]{km98} 
works also in positive characteristic. 
We conclude that the program ends with finite steps. 
\end{proof}

\section{BG inequality on surfaces} \label{sec:surf}
The goal of this section is to prove Theorem \ref{thm:bgintro}. 

\subsection{Statement of theorems}

We use the following terminology:
\begin{defin}
Let $S$ be a smooth projective surface. 
We say that a divisor $D$ on $S$ is 
a {\it birational pull-back of a very ample divisor} 
if there exists a birational morphism 
$\phi \colon S \to \overline{S}$ to 
a normal projective surface $\overline{S}$ 
and a very ample divisor $A$ on $\overline{S}$ 
such that $D=\phi^*A$. 
We define $\BD(S)$ to be a set of 
all birational pull-backs of very ample divisors. 
\end{defin}

We define a modified version of the discriminant as follows: 
\begin{defin} \label{defin:constant}
Let $T$ be a smooth projective surface. 
We define a constant $C_{[T]}$, which only depends on 
the birational equivalence class of $T$, as follows: 
\begin{enumerate}
\item If the Kodaira dimension $\kappa(S)=2$, 
let $S$ be the minimal model of $T$. 
We set 
\begin{align*}
&d_{[T]}:=\min\left\{K_S.H \colon H \in \BD(S),  H^2 \geq K_SH \right\},\\
&C_{[T]}:=d_{[T]}-\chi(\mcO_S)+2. 
\end{align*}
\item If $\kappa(T)=1$ and $T$ is quasi-elliptic, 
we set $C_{[T]}:=2-\chi(\mcO_T)$. 
\item Otherwise, we set $C_{[T]}:=0$. 
\end{enumerate}

We then define 
\[
\widetilde{\Delta}_{[T]}:=\Delta+C_{[T]}\ch_0^2. 
\]
\end{defin}

\begin{rmk} \label{rmk:constant.v.a.}
Let $S$ be a minimal surface of general type. 
Then the divisor $5K_S$ is very ample 
considered as a divisor on the canonical model of $S$. 
Hence we have $5K_S \in \BD(S)$ and the inequality 
$C_{[S]} \leq 5K_S^2-\chi(\mcO_S)+2$ 
holds. 
\end{rmk}

\begin{ex} \label{ex:hyper}
Let $S:=S^2_d \subset \bP^3$ be a hypersurface of degree $d \geq 5$. 
Denote by $H$ the restriction of a hyperplane class on $\bP^3$. 
Then we have $K_S=(d-4)H$ 
and it is easy to compute that 
\[
C_{[S^2_d]}=\frac{5}{6}d^3-7d^2+\frac{85}{6}d+2,
\]
which is positive for $d \geq 5$. 
\end{ex}

We prove the following results: 
\begin{thm} \label{thm:bg1}
Let $T$ be a smooth projective surface. 
Then for every numerically non-trivial nef divisor $L$ on $T$ 
and $\mu_L$-semistable torsion free sheaf $E \in \Coh(T)$, 
we have 
\[
\widetilde{\Delta}_{[T]}(E) \geq 0. 
\]
\end{thm}

\begin{thm} \label{thm:restr}
Let $T$ be a smooth projective surface. 
$L$ a birational pull-back of a very ample divisor. 
Let $E \in \Coh(T)$ be a $\mu_L$-semistable torsion free sheaf  on $T$. 
Assume that for any general member $C \in |L|$, 
the restriction $E|_D$ is not slope semistable. 
Then we have 
\[
\sum_{i<j}r_ir_j(\mu_i-\mu_j)^2 
\leq L^2 \cdot \widetilde{\Delta}_{[T]}(E), 
\]
where we denote by $r_i, \mu_i$ the ranks and the slopes 
of the Hardar-Narasimhan factors of $E|_C$. 
\end{thm}

When $\kappa(T) \leq 1$, 
the results essentially follows from \cite[Theorem 7.1]{lan16}. 
For surfaces of general type, 
we prove the above theorems at the same time 
by induction on $\ch_0(E)$, 
following the idea of Langer \cite{lan04} 
(see also \cite{blms17}). 
Note that the assertions are trivial 
when $\ch_0(E)=1$. 
Let $T^1(r)$ (resp. $T^2(r)$) be the statement that 
Theorem \ref{thm:bg1} (resp. \ref{thm:restr}) holds 
when $\ch_0(E) \leq r$. 
We prove that $T^2(r)$ implies $T^1(r)$, 
and that $T^1(r-1)$ implies $T^2(r)$.

\subsection{Invariance of the BG inequality}
In this subsection, 
we recall some results concerning 
the invariance of the BG inequality 
under the blow-ups and the change of polarizations, 
which are essentially proved in \cite{lan16}. 

\begin{lem} \label{lem:strss}
Let $T$ be a smooth projective surface, 
$L$ a numerically non-trivial nef divisor. 
Suppose that we have an exact sequence 
\[
0 \to E_1 \to E \to E_2 \to 0
\]
of torsion free sheaves on $T$ 
with $\mu_L(E_1)=\mu_L(E)=\mu_L(E_2)$. 
Then we have 
\[
\frac{\widetilde{\Delta}_{[T]}(E)}{\ch_0(E)} 
\geq \frac{\widetilde{\Delta}_{[T]}(E_1)}{\ch_0(E_1)}
+\frac{\widetilde{\Delta}_{[T]}(E_2)}{\ch_0(E_2)}. 
\]
\end{lem}
\begin{proof}
By the assumption, we have 
$\left(\frac{\ch_1(E_1)}{\ch_0(E_1)}-\frac{\ch_1(E_2)}{\ch_0(E_2)} \right)L=0$. 
Hence by the Hodge index theorem, 
we have 
\begin{align*}
\frac{\widetilde{\Delta}_{[T]}(E)}{\ch_0(E)} 
&=\frac{\widetilde{\Delta}_{[T]}(E_1)}{\ch_0(E_1)}
+\frac{\widetilde{\Delta}_{[T]}(E_2)}{\ch_0(E_2)} 
-\frac{\ch_0(E_1)\ch_0(E_2)}{\ch_0(E)}
\left(\frac{\ch_1(E_1)}{\ch_0(E_1)}-\frac{\ch_1(E_2)}{\ch_0(E_2)} \right)^2 \\
&\geq \frac{\widetilde{\Delta}_{[T]}(E_1)}{\ch_0(E_1)}
+\frac{\widetilde{\Delta}_{[T]}(E_2)}{\ch_0(E_2)}. 
\end{align*}
\end{proof}

\begin{prop}[{\cite[Proposition 6.2]{lan16}}] \label{prop:deform}
Let $T$ be a smooth projective surface, 
$r \geq 2$ an integer. 
Suppose that $T^1(r)$ holds 
for some numerically non-trivial nef divisor $L$. 
Then it also holds 
for every choice of numerically non-tirivial nef divisors $M$. 
\end{prop}
\begin{proof}
Let $E$ be a $\mu_M$-semistable torsion free sheaf 
with $\ch_0(E) \leq r$. 
We prove the assertion by induction on $\ch_0(E)$. 
For $\ch_0(E)=1$, the assertion is trivial, 
so assume that $2 \leq \ch_0(E) \leq r$. 
Let us put $M_t:=tM+(1-t)L$ for $t \in [0, 1]$. 
There are two possible cases. 
If $E$ is $\mu_L$-semistable, 
then the assertion follows from $T^1(r)$ for $L$. 

If $E$ is not $\mu_L$-semistable, 
then there exists a real number $t \in (0, 1)$ such that 
$E$ is strictly $\mu_{M_t}$-semistable. 
Hence there exists an exact sequence 
\[
0 \to E_1 \to E \to E_2 \to 0
\]
of torsion free $\mu_{M_t}$-semistable sheaves with 
$\mu_{M_t}(E_1)=\mu_{M_t}(E)=\mu_{M_t}(E_2)$. 
Now by the induction hypothesis and Lemma \ref{lem:strss}, 
we get the result. 
\end{proof}

\begin{prop}[{\cite[Proposition 6.5]{lan16}}] \label{prop:minimal}
Take an integer $r \geq 2$. 
Let $S$ be a smooth projective surface, 
$L$ a numerically non-trivial nef divisor on $S$. 
Let $\psi \colon T \to S$ be the blow-up at points. 
Suppose that $T^1(r)$ holds for $(S, L)$. 
Then it also holds for $(T, \psi^*L)$. 
\end{prop}
\begin{proof}
Let $E$ be a $\mu_{\psi^*L}$-semistable torsion free sheaf. 
By taking the double dual, we may assume that $E$ is locally free. 
Let $F:=\left(\psi_*E \right)^{**}$. 
Then $F$ is $\mu_L$-semistable. 
Moreover, by \cite[Lemma 6.4]{lan16}, 
we have $\Delta(F) \geq \Delta(E)$. 
Since we have $\ch_0(F)=\ch_0(E)$, 
the inequality 
$\widetilde{\Delta}_{[T]}(F)
\geq \widetilde{\Delta}_{[T]}(E)$ 
also holds, and hence the assertion holds. 
\end{proof}

\subsection{Proof of Theorems \ref{thm:bg1} and \ref{thm:restr}}

For surfaces with $\kappa(S) \leq 1$,
 we have the following result of Langer: 
\begin{thm}[{\cite[Theorem 7.1]{lan16}}] \label{thm:bgkleq1}
Let $T$ be a smooth projective surface
with $\kappa(S) \leq 1$, 
$L$ be a numerically non-trivial nef divisor on $T$. 
Then for every $\mu_L$-semistable torsion free sheaf $E$ on $T$, 
we have 
$\widetilde{\Delta}_{[T]}(E) \geq 0$. 
\end{thm}
\begin{proof}
We prove the assertion 
when $\kappa(S)=1$ and $S$ is quasi-elliptic. 
In this case, we have $K_S^2=0$ and $K_S$ is numerically non-trivial. 
Hence by Proposition \ref{prop:deform}, 
it is enough to prove the assertion for $\mu_{K_S}$-semistable sheaves $E$. 
If $\ch_0(E)=1$, then the result trivially holds. 
Hence we may assume that $\ch_0(E) \geq 2$. 
Now by Serre duality and the Riemann-Roch theorem, we have 
\begin{align*}
\chi(\mcO_S)\ch_0(E)^2-\Delta(E)
&=\chi(E, E) \\
&\leq \hom(E, E)+\hom(E, E(K_S)) 
\leq 2, 
\end{align*}
and the desired inequality holds. 

For other cases, the result is proved in \cite[Theorem 7.1]{lan16}. 
\end{proof}

For surfaces of general type, 
the idea of the proof is similar as above. 
However, we need some more works 
to bound $\hom(E, E(K_S))$. 

We use the following easy lemma: 
\begin{lem} \label{lem:hombound}
Let $C$ be a smooth projective curve, 
$E$ a slope semistable vector bundle of rank $r$. 
For every integer $d \geq 0$, we have 
\[
\hom\left(E, E(d) \right) \leq (d+1)r^2. 
\]
\end{lem}
\begin{proof}
We prove it by induction on $d \geq 0$. 
For $d=0$, the assertion follows from semistability of $E$. 
Let us consider the case when $d \geq 1$. 
Pick a point $p \in C$. 
By applying the functor $\Hom(E, -)$ 
to the exact sequence 
\[
0 \to E(d-1) \to E(d) \to E(d)|_p \to 0, 
\]
we have 
\begin{align*}
\hom\left(E, E(d) \right) 
&\leq \hom\left(E, E(d-1) \right)+\hom\left(E, E(d)|_p \right) \\
&=\hom\left(E, E(d-1) \right)+r^2, 
\end{align*}
and hence the assertion holds. 
\end{proof}

Now we are ready to prove Theorems \ref{thm:bg1} and \ref{thm:restr}: 
\begin{proof}[$T^1(r-1)$ implies $T^2(r)$]
This direction follows from the arguments 
as in \cite[Subsection 3.9]{lan04}. 
For the completeness, we include the proof here. 
Let $\Lambda \subset |L|$ be a general pencil, 
let $\widetilde{T}:=\left\{(t, C) \in T \times \Lambda \colon t \in C \right\}$ 
be the incidence variety.  
Let us put $d:=L^2$. 
Since $\Lambda$ is general, 
the base locus $\Bs(\Lambda)$ 
consists of $d$-distinct points, 
and the projection 
$q \colon \widetilde{T} \to T$ is 
the blow-up at $\Bs(\Lambda)$. 
Denote by $p \colon \widetilde{T} \to \Lambda$ the projection, 
and by $f$ a class of $p$-fiber. 

Let $E \in \Coh(T)$ be a $\mu_L$-semistable tosion free sheaf 
with $\ch_0(E)=r$, 
suppose that the restriction $E|_C$ is not slope semistable 
for general $C \in \Lambda$. 
Let $E_{\bullet} \subset q^*E$ be 
the $p$-relative Hardar-Narasimhan (HN) filtration, 
which is same as the HN filtration 
with respect to $\mu_f$-stability. 
As $f$ is a numerically non-trivial nef divisor, 
we can apply $T^1(r-1)$ to the factors $F_i:=E_i/E_{i-1}$ 
to get the inequality 
\begin{equation} \label{eq:bgFi}
\widetilde{\Delta}_{\left[\widetilde{T}\right]}(F_i) \geq 0. 
\end{equation}
Note also that as $\widetilde{T} \to T$ is the blow-up, 
we have 
$\widetilde{\Delta}_{[\widetilde{T}]}
=\widetilde{\Delta}_{[T]}$. 

Let $N_1, \cdots, N_d$ be 
the $q$-exceptional divisors. 
Then there exist divisors $M_i$ on $T$ 
and integers $b_{ij}$ such that 
$\ch_1(F_i)=q^*M_i+\sum_jb_{ij}N_j$. 
Put $b_i:=\sum_jb_{ij}/d$. 
Then we have 
\begin{equation}\label{eq:mui}
\mu_i=\frac{f\ch_1(F_i)}{r_i}
=\frac{M_iL+b_id}{r_i}. 
\end{equation}

On the other hand, 
as $E$ is $\mu_L$-semistable and $q_*E_i \subset E$, 
we have 
\[
\frac{\sum_{j \leq i}M_jL}{\sum_{j \leq i}r_j} \leq \mu_L(E)
\]
and hence 
\begin{equation}\label{eq:bjd}
\sum_{j \leq i}b_jd \geq \sum_{j \leq i}r_j(\mu_j-\mu_L(E)). 
\end{equation}

Now using the inequality (\ref{eq:bgFi}), 
we have 
\begin{equation} \label{eq:surface-evaluate}
\begin{aligned}
\frac{d\widetilde{\Delta}_{[T]}(E)}{r}
&=\frac{\sum_id\widetilde{\Delta}_{[T]}(F_i)}{r_i}
-\frac{d}{r}\sum_{i <j}r_ir_j\left(\frac{\ch_1(F_i)}{r_i}-\frac{\ch_1(F_j)}{r_j} \right)^2 \\
&\geq \frac{d}{r}\sum_{i<j}r_ir_j\left(
	d\left(\frac{b_i}{r_i}-\frac{b_j}{r_j} \right)^2 
	-\left(\frac{M_i}{r_i}-\frac{M_j}{r_j} \right)^2
	\right) \\
&\geq \frac{1}{r}\sum_{i<j}r_ir_j\left(
	d^2\left(\frac{b_i}{r_i}-\frac{b_j}{r_j} \right)^2 
	-\left(\frac{M_iL}{r_i}-\frac{M_jL}{r_j} \right)^2
	\right), 
\end{aligned}
\end{equation}
where the last inequality follows from the Hodge index theorem. 
By using the equation (\ref{eq:mui}), 
the bottom line of the above inequalities becomes 
\[
2\sum_idb_i\mu_i-\frac{1}{r}\sum_{i<j}r_ir_j(\mu_i-\mu_j)^2. 
\]
By the inequality (\ref{eq:bjd}), we get 
\begin{align*}
\sum_idb_i\mu_i
&=\sum_i\left(\sum_{j \leq i}db_j \right)(\mu_i-\mu_{i+1}) \\
&\geq \sum_i\left(\sum_{j \leq i}r_j(\mu_j-\mu_L(E)) \right)(\mu_i-\mu_{i+1}) \\
&=\sum_ir_i\mu_i^2-r\mu_L(E)^2
=\sum_{i<j}\frac{r_ir_j}{r}(\mu_i-\mu_j)^2, 
\end{align*}
and hence $T^2(r)$ holds. 
\end{proof}

\begin{proof}[$T^2(r)$ implies $T^1(r)$]
By Theorem \ref{thm:bgkleq1}, 
we may assume that $\kappa(T)=2$. 
Moreover, by Propositions \ref{prop:deform} and \ref{prop:minimal}, 
it is enough to prove the assertion for a minimal surface $S$, 
and a divisor $L \in \BD(S)$ with $K_SL=d_{[S]}$. 

Assume for a contradiction that 
there exists a $\mu_L$-semistable torsion free sheaf $E \in \Coh(S)$ 
with $\ch_0(E)=r$ such that $\widetilde{\Delta}_{[S]}(E) < 0$. 
Then, by $T^2(r)$, the restriction $E|_C$ is slope semistable 
for a general member $C \in |L|$. 

Note that we may assume that $E$ is $\mu_L$-stable. 
By Serre duality and the Riemann-Roch theorem, we have 
\begin{equation} \label{eq:rr}
\begin{aligned}
\chi(\mcO_S)r^2-\Delta(E)
&=\chi(E, E) \\
&\leq \hom(E, E)+\hom(E, E(K_S)) \\
&=1+\hom(E, E(K_S)). 
\end{aligned}
\end{equation}

On the other hand, by the exact sequence 
\[
0 \to E(K_S-L) \to E(K_S) \to E(K_S)|_C \to 0, 
\]
Lemma \ref{lem:hombound}, 
and the equation $K_SL=d_{[S]}$, 
we have 
\begin{align*}
\hom(E, E(K_S)) 
&\leq \hom(E, E(K_S-L))+\hom(E, E(K_S)|_C) \\
&\leq 1+\hom\left(E|_C, E(K_S)|_C \right) \\
&\leq 1+(d_{[S]}+1)r^2. 
\end{align*}

Combining with the inequality (\ref{eq:rr}), 
we obtain 
\[
\Delta(E)+\left(d_{[S]}-\chi(\mcO_S)+1 \right)r^2+2 \geq 0, 
\]
which is a contradiction. 
\end{proof}

\begin{rmk} \label{rmk:rank2}
Let $S$ be a minimal surface of general type. 
For rank two slope semistable bundles $E$ on $S$, 
a similar bound is obtained in \cite[Theorem 12]{sb91}. 
In fact, we have 
\[
\begin{cases}
\Delta(E)+K_S^2 \geq 0 & (\cha(k)=p \geq 3), \\
\Delta(E)+\max\left\{K_S^2, K_S^2-3\ch(\mcO_S)+2 \right\} \geq 0 & (p=2). 
\end{cases}
\]
\end{rmk}

\begin{ex}[Counter-examples to Kodaira vanishing]
It is known that there exist 
a surface $S$ of general type 
and an ample divisor $L$ on $S$ such that 
$H^1(S, L^{-1}) \neq 0$ 
(cf. \cite{muk13, ray78}). 
Our Theorem \ref{thm:bg1} gives 
an upper bound 
on the intersection number $L^2$ 
for such a divisor $L$. 
Indeed, consider the non-trivial extension 
\[
0 \to \mcO_S \to E \to L \to 0, 
\]
which exist as we assume $H^1(S, L^{-1}) \neq 0$. 
Then the bundle $E$ is $\mu_L$-stable, 
and we have $\Delta(E)=-L^2$. 
Hence by Theorem \ref{thm:bg1}, we have 
\[
4\left(5K_S^2-\chi(\mcO_S)+2 \right) 
\geq 4C_{[S]} \geq L^2. 
\]

See \cite{sb91} (and Remark \ref{rmk:rank2} above) 
for the detailed study on the vanishing theorem 
in positive characteristic. 
\end{ex}

\section{BG inequality in higher dimension} \label{sec:high}
Let $X$ be a smooth projective variety of dimension $n \geq 3$, 
and $H$ a very ample divisor. 
Let $\Pi:=|H|$ be the complete linear system, 
and define the incidence variety $\mcS$ as 
\[
\mcS:=\left\{
	\left(x, (H_i) \right) \in X \times \Pi^{\times (n-2)} \colon 
	x \in H_i \mbox{ for all } i
	\right\}. 
\]
We have the following diagram: 
\[
\xymatrix{
&\mcS \ar[r]^<<<<{p} \ar[d]_{q} &\Pi^{\times (n-2)} \\
&X. &
}
\]

For a point $t \in \Pi^{\times n-2}$, 
we denote by $\mcS_t:=p^{-1}(t)$ 
the scheme-theoretic fiber of $t$. 
First we need to define a well-defined discriminant with a modification term. 

\begin{def-prop} \label{def-prop:const}
There exists a non-empty open subset 
$U \subset \Pi^{\times (n-2)}$ satisfying the following properties: 
\begin{enumerate}
\item For every point $t \in U$, the fiber $\mcS_t$ is smooth, 
\item The Kodaira dimension $\kappa(\mcS_t)$ is constant for all $t \in U$, 
\item When $\kappa(\mcS_t) \geq 0$, let $R_t$ be the minimal model of $\mcS_t$. 
Then the invariants $K_{R_t}^2, \chi(\mcO_{R_t})$ are constant for all $t \in U$. 
\end{enumerate}

In particular, the number 
\[
C_{X, H}:=
\begin{cases}
0 & (\kappa(\mcS_t) \leq0), \\
5K_{R_t}^2-\chi(\mcO_{R_t})+2 & (\kappa(\mcS_t) \geq 1) 
\end{cases}
\]
is independent on the points $t \in U$. 
Moreover, we have 
$C_{X, T} \geq C_{[\mcS_t]}$ for all $t \in U$, 
where the constant $C_{[\mcS_t]}$ 
is defined as in Definition \ref{defin:constant}. 
\end{def-prop}
\begin{proof}
First take an open subset $U' \subset \Pi^{\times (n-2)}$ 
so that all the fibers $\mcS_t$ are smooth, 
and put $\mcS_{U'}:=p^{-1}(U')$. 
It is enough to show the following claim: 
Let $C \to U'$ be a non-constant morphism 
from a smooth quasi-projective curve $C$, 
and set $\mcS_C:=C \times_{U'} \mcS_{U'}$. 
Then for a general point $c \in C$, 
the Kodaira dimension $\kappa(\mcS_c)$ 
and the invariants 
$K_{\mcS_c}^2, \chi(\mcO_{\mcS_c})$ are constant. 

Note that the morphism $\mcS_C \to C$ is 
a smooth projective morphism of relative dimension two. 
By Theorem \ref{thm:mmp}, 
we can run an MMP for $\mcS_C$ over $C$: 
\[
\mcS_C=R_0 \dashrightarrow R_1 \dashrightarrow \cdots \dashrightarrow R_l. 
\]
Observe the following facts: 
\begin{itemize}
\item If $R_i \dashrightarrow R_{i+1}$ contracts a divisor to a point or is a flip, 
then it does not change the general fiber of $R_i \to C$. 
\item If $R_i \dashrightarrow R_{i+1}$ contracts a divisor to a curve, 
then it contracts (-1)-curves of the general fiber of $R_i \to C$. 
\end{itemize}
Hence for general $t \in C$, the morphism 
$\mcS_t \to (R_l)_t$ contracts the same number of $(-1)$-curves. 
Moreover, the invariants 
$K_{\mcS_t}^2, \chi(\mcO_{\mcS_t})$ 
are constant for all $t \in C$. 
Hence the same holds for the minimal models $(R_l)_t$ with $t \in C$ general. 

So we have proven the existence of 
an open subset $U \subset \Pi^{\times (n-2)}$ 
with the required properties. 
For the inequality $C_{X, H} \geq C_{[\mcS_t]}$, 
see Remark \ref{rmk:constant.v.a.}. 
\end{proof}

Following \cite{blms17}, we use the following notion: 
\begin{defin}
\begin{enumerate}
\item For general hyperplanes $H_1, H_2 \in |H|$, 
we call the intersection $H_1 \cap H_2$ 
as a {\it codimension two linear subspace of $X$}. 

\item Let $\mcV=(V_1, \cdots, V_m)$ be an ordered configuration of codimension two 
linear subspaces in $X$. 
We say that $\psi \colon Y \to X$ is a {\it blow-up along $\mcV$} 
if $\psi$ is the iterated blow-up along the strict transforms of the $V_i$'s. 
\end{enumerate}
\end{defin}

For the blow-up $\psi \colon Y \to X$ 
along an ordered configuration of codimension two linear subspaces, 
we define a discriminant $\widetilde{\Delta}_{Y, \psi^*H}$ as 
\[
\widetilde{\Delta}_{Y, \psi^*H}:=\psi^*H^{n-2}\Delta+C_{X, H}\ch_0^2. 
\]

\begin{thm} \label{thm:BGhigh}
Let $\psi \colon Y \to X$ be the blow-up of 
an ordered configuration of codimension two linear subspaces. 
Let $E \in \Coh(Y)$ be a $\psi^*H$-slope semistable torsion free sheaf.  
Then we have 
\[
\widetilde{\Delta}_{Y, \psi^*H}(E) \geq 0. 
\]
\end{thm}

\begin{thm} \label{thm:restrhigh}
Let $\psi \colon Y \to X$ be the blow-up of 
an ordered configuration of codimension two linear subspaces. 
Let $E \in \Coh(Y)$ be a $\psi^*H$-slope semistable torsion free sheaf. 
Assume that for any general member $D \in |\psi^*H|$, 
the restriction $E|_D \in \Coh(D)$ is not $\psi^*H|_D$-semistable. 
Let $r_i, \mu_i$ be the ranks and slopes of its Hardar-Narasimhan factors. 
Then we have 
\[
\sum_{i<j}r_ir_j(\mu_i-\mu_j)^2 
\leq H^n \cdot \widetilde{\Delta}_{Y, \psi^*H}(E). 
\]
\end{thm}

Let $T^1(r)$ (resp. $T^2(r)$) be the statement that 
Theorem \ref{thm:BGhigh} (resp. \ref{thm:restrhigh}) holds when $\ch_0(E) \leq r$. 
As in the surface case, 
we prove that $T^2(r)$ implies $T^1(r)$, 
and that $T^1(r-1)$ implies $T^2(r)$. 

\begin{proof}[$T^1(r-1)$ implies $T^2(r)$]
The proof is almost same as the surface case. 
We take a general pencil $\Lambda \subset |\psi^*H|$ 
and consider the incidence variety 
\begin{equation} \label{eq:incidence}
\widetilde{Y}:=\left\{
	(y, D) \in Y \times \Lambda \colon y \in D
	\right\}. 
\end{equation}
Denote by $q \colon \widetilde{Y} \to Y$, 
$p \colon \widetilde{Y} \to \Lambda$ the projections, 
and by $f$ the class of a $p$-fiber. 
Note that the composition 
$\widetilde{\psi} \colon \widetilde{Y} \to Y \to X$ 
is again the blow-up of an ordered configuration 
of codimension two linear subspaces. 
Let $E_\bullet \subset q^*E$ be 
the $p$-relative HN filtration, 
which is same as the HN filtration with respect to 
$f.\widetilde{\psi}^*H^{n-2}$-slope stability. 
By Proposition \ref{prop:changepol} below, 
each HN factor $F_i:=E_i/E_{i-1}$ 
satisfies the inequality 
$\widetilde{\Delta}_{Y, \psi^*H}(F_i) \geq 0$. 
Now the remaining computations are same as in the surface case. 
\end{proof}

\begin{prop}[{\cite[Proposition 8.9]{blms17}}]
\label{prop:changepol}
Suppose $T^1(r-1)$ holds. 
Let $F \in \Coh(\widetilde{Y})$ be 
a $f.\widetilde{\psi}^*H^{n-2}$-slope semistable torsion free sheaf 
with $\ch_0(F) \leq r-1$. 
Then the inequality 
\[
\widetilde{\Delta}_{\widetilde{Y}, \widetilde{\psi}^*H}(F) \geq 0
\]
holds. 
\end{prop}
\begin{proof}
The proof of \cite[Proposition 8.9 ]{blms17} works, 
as our modified discriminant differs from the usual one 
only by a term $\ch_0^2$. 
\end{proof}

\begin{proof}[$T^2(r)$ implies $T^1(r)$]
Let $E \in \Coh(Y)$ be a $\psi^*H$-semistable torsion free sheaf with $\ch_0(E)=r$. 
Assume for a contradiction that we have $\Delta_{Y, \psi^*H}(E) <0$. 
As we assume $T^2(r)$, the restriction $E|_D$ to 
a general member $D \in |\psi^*H|$ 
is $\psi^*H|_D$-slope semistable. 
Note that by taking it general, we may assume that 
$D$ is again the blow-up of an ordered configuration codimension two linear subspaces 
of a general member of $|H|$. 
Hence by induction on $n=\dim X$, 
we may assume that $X$ is a surface. 
Since we have $C_{X, H} \geq C_{[X]}$ 
(see Definition-Proposition \ref{def-prop:const}), 
we get a contradiction by Theorem \ref{thm:bg1}. 
\end{proof}

\section{BG inequality on hypersurfaces}
\label{sec:hyper}
The goal is to prove a variant of the BG inequality 
for the $H$-discriminant on hypersurfaces in the projective spaces. 
Let $S^n_d \subset \bP^{n+1}$ be a smooth hypersurface 
of dimension $n \geq 2$, degree $d \geq 1$. 
Denote by $H$ the restriction of the hyperplane class of $\bP^{n+1}$. 
We prepare several notations: 
\begin{defin}
Let $\psi \colon Y \to S^n_d$ be the blow-up 
of an ordered configuration of codimension two linear subspaces. 
Denote by $E_1, \cdots, E_l$ the $\psi$-exceptional prime divisors. 
\begin{enumerate}
\item We denote by $\Lambda_Y:=\bZ \oplus \bigoplus_{i}\bZ[E_i]$. 
We define a quadratic form $q_Y$ on $\Lambda_Y$ as follows: 
\[
q_Y\left(b, \sum_ia_iE_i \right):=b^2+H^n\psi^*H^{n-2}\left(\sum_ia_iE_i \right)^2. 
\]
\item Identify $\NS(Y)$ with $\psi^*\NS(S^n_d)\oplus \bigoplus_i \bZ[E_i]$. 
Then we define a group homomorphism 
$\pr_H \colon \NS(Y) \to \Lambda_Y$ as follows: 
\begin{align*}
\pr_H\left(\psi^*M, \sum_ia_iE_i \right):=
\left(H^{n-1}M, \sum_ia_iE_i \right). 
\end{align*}
We denote by 
$\overline{\ch}_1^H:=\pr_H \circ \ch_1 \colon  
K(Y) \rightarrow \NS(Y) \to \Lambda_Y$ 
the composition. 

\item For a coherent sheaf $E \in \Coh(Y)$, we define 
\[
Q_{Y, \psi^*H}(E):=q_Y\left(\overline{\ch}^H_1(E) \right)-2H^n\ch_0(E)\psi^*H^{n-2}\ch_2(E). 
\]
\end{enumerate}
\end{defin}

\begin{rmk}
The quadratic form $Q_{Y, \psi^*H}$ is a natural generalization 
of the usual discriminant. Indeed, we have the following:
\begin{enumerate}
\item We have $Q_{S^n_d, H}=\overline{\Delta}_H$. 
\item If $\Pic(S^n_d)=\bZ[H]$, then we have 
$Q_{Y, \psi^*H}=H^n \cdot H^{n-2}\Delta$. 
\end{enumerate}
\end{rmk}

As similar to the previous sections, 
we prove the following results at the same time: 
\begin{thm} \label{thm:hyper}
Let $\psi \colon Y \to S^n_d$ be the blow-up 
of an ordered configuration of codimension two linear subspaces. 
For every $\mu_{\psi^*H}$-semistable 
torsion free sheaf $E \in \Coh(Y)$, we have 
\[
Q_{Y, \psi^*H}(E) \geq 0. 
\]
\end{thm}

\begin{thm} \label{thm:hyperrestr}
Let $\psi \colon Y \to S^n_d$ be the blow-up 
of an ordered configuration of codimension two linear subspaces. 
Let $E \in \Coh(Y)$ be a $\psi^*H$-slope semistable torsion free sheaf. 
Assume that for any general member $D \in |\psi^*H|$, 
the restriction $E|_D \in \Coh(D)$ is not $\psi^*H|_D$-semistable. 
Let $r_i, \mu_i$ be the ranks and slopes of its Hardar-Narasimhan factors. 
Then we have 
\[
\sum_{i<j}r_ir_j(\mu_i-\mu_j)^2 
\leq Q_{Y, \psi^*H}(E). 
\]
\end{thm}

\subsection{The case of $n=2$}
The key input is the following result: 
\begin{thm}[{\cite[Corollary 6.4]{kos20b}}]
\label{thm:hyper2}
Take an integer $d \geq 1$. 
For every torsion free 
$\mu_H$-semistable sheaf $E \in \Coh(S^2_d)$, 
we have 
\[
Q_{S^2_d, H}(E)=\overline{\Delta}_H(E) \geq 0. 
\]
\end{thm}

For blow-ups, we use the following easy lemma: 
\begin{lem} \label{lem:blch}
Let $\psi \colon Y \to S^2_d$ be the blow-up at points. 
Let $E \in \Coh(Y)$ be a torsion free sheaf. 
The following statements hold: 
\begin{enumerate}
\item The sheaf $\psi_*E \in \Coh(S^2_d)$ is torsion free, and 
the sheaf $\dR^1\psi_*(E)$ is supported on a zero dimensional subscheme. 
\item We have $\ch_0(\psi_*E)=\ch_0(E)$. 
\item We have $\psi^*H\ch_1(E)=H\ch_1(\psi_*E)$. 
\item If $E$ is $\mu_{\psi^*H}$-semistable, then 
$\psi_*E$ is $\mu_H$-semistable. 
\end{enumerate}
\end{lem}

\begin{prop} \label{prop:hyper2}
Let $\psi \colon Y \to S^2_d$ be the blow-up at $l$-distinct points. 
For every $\mu_{\psi^*H}$-semistable torsion free sheaf, we have 
\[
Q_{Y, \psi^*H}(E) \geq 0. 
\] 
\end{prop}
\begin{proof}
Let $E_1, \cdots, E_l$ be the $\psi$-exceptional divisors. 
There exists a line bundle $M$ on $S^2_d$ 
and integers $a_i \in \bZ$ such that 
$\ch_1(E)=\psi^*M+\sum_ia_iE_i$. 
Observe that the number $Q_{Y, \psi^*H}(E)$ 
is invariant under tensoring line bundles $\mcO(E_i)$, 
hence we may assume that $0 \leq a_i < \ch_0(E)$. 
Now we obtain 
\begin{align*}
0 \leq Q_{S^2_d, H}(\psi_*E) 
&\leq Q_{S^2_d, H}(\dR\psi_*E) \\
&=(HM)^2-2H^2\ch_0(E)\left(\ch_2(E)+\frac{\sum_ia_i}{2} \right) \\
&\leq (HM)^2-H^2\sum_ia_i^2-2H^2\ch_0(E)\ch_2(E)
=Q_{Y, \psi^*H}(E). 
\end{align*}
For the first inequality, we use Theorem \ref{thm:hyper2} 
and Lemma \ref{lem:blch} (4); 
for the second inequality, we use Lemma \ref{lem:blch} (1); 
the third equality follows from the Grothendieck-Rimann-Roch 
theorem for the morphism $p$, 
together with Lemma \ref{lem:blch} (2), (3); 
the fourth inequality follows from 
the assumption that $0 \leq a_i < \ch_0(E)$. 
\end{proof}

\subsection{The case of $n \geq 3$}
For higher dimension, 
we need a variant of Proposition \ref{prop:changepol}. 
Let $\psi \colon Y \to S^n_d$ be the blow-up of 
an ordered configuration of codimension two linear subspaces. 
Let $\Lambda \subset |\psi^*H|$ be a general pencil 
and consider the incidence variety $\widetilde{Y}$
as in (\ref{eq:incidence}). 
Denote by $q \colon \widetilde{Y} \to Y$, 
$p \colon \widetilde{Y} \to \Lambda$ the projections. 
Let $e$ be the $q$-exceptional divisor and $f$ the $p$-fiber. 
Note that, by taking the pencil $\Lambda$ general, 
the composition $\widetilde{\psi} \colon \widetilde{Y} \to Y \to S^n_d$ 
is again the blow-up along 
an ordered configuration of codimension two linear subspaces. 
Put $h:=\widetilde{\psi}^*H$. 

For a real number $t \geq 0$, 
we put $h_t:=th+f$. 
Define a group homomorphism 
$Z_{t} \colon 
H^0(\widetilde{Y}, \bQ) \oplus \Lambda_{\widetilde{Y}}  \to \bC$ as follows: 
\[
Z_{t}\left(r, b, \sum_ia_iE_i \right):=-(t+1)b
+h^{n-2}f \left(\sum_ia_iE_i \right)
+\sqrt{-1}r. 
\]
We denote as $\mu_{Z_{t}}:=-\Re Z_{t}/\Im Z_{t}$. 
For a coherent sheaf $E \in \Coh(\widetilde{Y})$, we have 
\[
\mu_{h^{n-2}h_t}(E)=\mu_{Z_t}\left(\ch_0(E), \overline{\ch}^H_1(E) \right). 
\]

\begin{lem} \label{lem:q-neg}
Let the notations as above. 
Then $Q_{\widetilde{Y}, h}$, considered as 
a quadratic form on $(\Lambda_{\widetilde{Y}})_\bR$, 
is negative semi-definite on the kernel 
$\Ker (Z_{t}) \subset 
H^0(\widetilde{Y}, \bQ) \oplus \Lambda_{\widetilde{Y}}$. 
\end{lem}
\begin{proof}
Take an element $\left(r, b, \sum_ia_iE_i \right) \in \Ker(Z_t)$. 
Then we have $r=0$. 

On the other hand, 
let us put 
\[
\overline{h}_t:=\pr_H(h_t)=(t+1)h^n-e, 
\]
where $e$ denotes the $q$-exceptional divisor. 
Then we have 
\begin{align*}
&q_{\widetilde{Y}}(\overline{h}_t)=(t+1)^2(H^n)^2-(H^n)^2 \geq 0, \\ 
&q_{\widetilde{Y}}\left(\overline{h}_t, \left(b, \sum_ia_iE_i \right) \right)
=H^n\left(-\Re Z_t\left(b, \sum_ia_iE_i \right) \right)=0. 
\end{align*}
Since $q_{\widetilde{Y}}$ has signature $(1, l)$ 
for some $l \geq 1$, we get 
\[
Q_{\widetilde{Y}, h}\left(r, b, \sum_ia_iE_i \right)
=q_{\widetilde{Y}}\left( b, \sum_ia_iE_i \right) \leq 0
\]
for $t >0$. 
When $t=0$, as $\overline{h}_0$ is 
a limit of $\overline{h}_t$, 
the assertion also holds. 
\end{proof}

\begin{prop} \label{prop:changepol-bar}
Fix a positive integer $r \geq 2$. 
Suppose that Theorem \ref{thm:hyper} holds 
for $\ch_0(E) \leq r-1$. 
Let $F \in \Coh(\widetilde{Y})$ be a torsion free 
$f.\widetilde{\psi}^*H^{n-2}$-slope semistable sheaf 
with $\ch_0(F) \leq r-1$. 
Then we have 
\[
Q_{\widetilde{Y}, \widetilde{\psi}^*H}(F) \geq 0. 
\]
\end{prop}
\begin{proof}
The proof of \cite[Proposition 8.9]{blms17} works 
by Lemma \ref{lem:q-neg} above. 
\end{proof}

\begin{proof}[Proof of Theorems \ref{thm:hyper} and \ref{thm:hyperrestr}]
Now we can check that the same proofs as in the previous sections work, 
using Propositions \ref{prop:hyper2} and \ref{prop:changepol-bar}. 
\end{proof}

\section{Construction of Bridgeland stability conditions} \label{sec:stab}
In this section, we construct Bridgeland stability conditions on surfaces, 
defined over an algebraically closed field of positive characteristic. 
We refer to papers \cite{bms16, bmt14a, bri07, bri08} 
and a lecture note \cite{ms17} for the basics of the theory of stability conditions. 

Let us first recall the definition of Bridgeland stability conditions \cite{bri07}: 

\begin{defin}
Let $\mcA$ be an abelian category. 
\begin{enumerate}
\item A group homomorphism 
$Z \colon K(\mcA) \to \bC$ 
is called a {\it stability function} if 
we have 
$Z\left(\mcA \setminus \{0\} \right) \subset \mcH \cup \bR_{<0}$, 
where we denote by $\mcH$ the upper half plane. 
\item For a stability function $Z \colon K(\mcA) \to \bC$ 
and an object $E \in \mcA$, 
we define the {\it $Z$-slope} of $E$ as 
\[
\mu_Z(E):=-\frac{\Re Z(E)}{\Im Z(E)}. 
\]
\item We say that an object 
$E \in \mcA$ is {\it $Z$-(semi)stable} if 
for every non-zero proper subobject 
$0 \neq F \subsetneq E$, 
we have an inequality $\mu_Z(F) \leq \mu_Z(E)$. 
\end{enumerate}
\end{defin}

\begin{defin}
Let $\mcD$ be a triangulated category. 
A {\it stability condition} on $\mcD$ is a pair $(Z, \mcA)$ 
consisting of the heart 
$\mcA \subset \mcD$ 
of a bounded t-structure 
and a group homomorphism 
$Z \colon K(\mcA) \to \bC$ 
satisfying the following axioms: 
\begin{enumerate}
\item The group homomorphism 
$Z \colon K(\mcA) \to \bC$ 
is a stability function. 
\item It satisfies the {\it Hardar-Narasimhan (HN)} property, i.e., 
for every object $E \in \mcA \setminus \{0\}$, 
there exists a Hardar-Narasimhan filtration of $E$ 
with respect to $Z$-stability. 
\end{enumerate}
We also call a homomorphism $Z$ as 
a {\it central charge}. 
\end{defin}

\begin{defin}
Let $\mcD$ be a triangulated category. 
Fix a finitely generated free abelian group $\Lambda$ 
and a group homomorphism 
$\cl \colon K(\mcD) \to \Lambda$. 
We say that a stability condition $(Z, \mcA)$ on $\mcD$ 
satisfies the {\it support property with respect to $(\Lambda, \cl)$} if 
the central charge $Z$ factors through $\cl$, i.e., 
$Z \colon K(\mcD) \to \Lambda \to \bC$, and 
there exists a quadratic form $Q$ on $\Lambda \otimes_{\bZ} \bR$ 
such that 
\begin{enumerate}
\item The kernel $\Ker Z \subset \Lambda$  is $Q$-negative definite, 
\item For every $Z$-semistable object $0 \neq E \in \mcA$, 
we have $Q(\cl(E)) \geq 0$. 
\end{enumerate}
\end{defin}

In the following, we explain the construction 
of Bridgeland stability conditions on 
the derived category $D^b(S)$ of 
a smooth projective surface $S$. 
We fix a lattice 
$H^*_{alg}(S):=\Image\left(\ch \colon K(X) \to H^{2*}(X, \bQ_\ell) \right)$. 

We use the following central charge function: 
\begin{defin} 
Let $S$ be a smooth projective surface, 
$H$ an ample $\bR$-divisor 
and $B$ an arbitrary $\bR$-divisor. 
Let $C_{[S]}$ be a constant defined in 
Definition \ref{defin:constant}. 
We define a group homomorphism 
$Z_{H, B}^{[S]} \colon K(X) \to \bC$ as follows: 
\[
Z_{H, B}^{[S]}:=
-\ch_2^B+\left(\frac{C_{[S]}}{2H^2}+\frac{1}{2} \right)H^2\ch_0^B
+\sqrt{-1}H\ch_1^B. 
\]
\end{defin}

For the heart, we need the theory of {\it torsion pairs} and {\it tilting} 
(see \cite{hrs96}). 
Let us fix an ample $\bR$-divisor $H$ on $S$, 
and an arbitrary $\bR$-divisor $B$. 
We define full subcategories 
$\mcT_{H, B}, \mcF_{H, B} \subset \Coh(S)$ as follows: 
\begin{align*}
&\mcT_{H, B}:=\left\langle
	T \in \Coh(S) \colon T \mbox{ is } \mu_H \mbox{-semistable with } \mu_H(T) > HB
	\right\rangle \\
&\mcF_{H, B}:=\left\langle
	F \in \Coh(S) \colon F \mbox{ is } \mu_H \mbox{-semistable with } \mu_H(F) \leq HB
	\right\rangle, 
\end{align*}
where, for a set $S \subset \Coh(S)$, 
we denote by $\langle S \rangle \subset \Coh(S)$ the extension closure. 
We then define the {\it tilted-heart} as 
the extension closure in the derived category: 
\[
\Coh^{H, B}(S):=\langle \mcF_{H, B}[1], \mcT_{H, B} \rangle 
\subset D^b(S). 
\]

We need the following lemma: 
\begin{lem}[{\cite[Corollary 7.3.3]{bmt14a}}]
\label{lem:support}
Let $H$ be an ample divisor on $S$. 
There exists a constant $C_H \geq 0$ such that 
for every effective divisor $D$ on $S$, we have 
\[
C_H(HD)^2+D^2 \geq 0. 
\]
\end{lem}

\begin{thm} \label{thm:stab}
Let $S$ be a smooth projective surface 
defined over an algebraically closed field of positive characteristic. 
Let $H$ be an ample $\bR$-divisor 
and $B$ an arbitrary $\bR$-divisor.  
Let $C_{[S]}$ be the constant defined as in Definition \ref{defin:constant}. 
Then the pair $\left(Z_{H, B}^{[S]}, \Coh^{H, B}(S) \right)$ 
defines a Bridgeland stability condition on $D^b(S)$ 
and satisfies the support property with respect to $(\ch, H^*_{alg}(S))$. 
\end{thm}
\begin{proof}
We check that the following condition holds: 
Let $E \in \Coh(S)$ be a $\mu_H$-semistable torsion free sheaf 
with $\mu_H(E)=HB$. 
Then we have $\Re Z_{\alpha, \beta}^{[S]}(E) > 0$. 
Indeed, by Theorem \ref{thm:bg1}, we have 
\begin{align*}
0 &\leq \Delta(E)+C_{[S]}\ch_0(E)^2 
\leq \overline{\Delta}_{H, B}(E)+C_{[S]}\ch_0(E)^2 \\
&=\ch_0^B(E)\left(-2H^2\ch^B_2(E)+C_{[S]}\ch_0^B(E) \right), 
\end{align*}
where we define 
$\overline{\Delta}_{H, B}:=\left(H\ch_1^B \right)^2-2H^2\ch_0^B\ch_2^B$, 
and the second inequality follows from 
the Hodge index theorem. 
Hence the desired inequality $\Re Z_{H, B}^{[S]}(E) > 0$ holds.  
Now we can show that the pair 
$\left(Z_{H, B}^{[S]}, \Coh^{H, B}(S) \right)$ 
defines a stability condition 
by the standard arguments as in \cite{ab13, bri08}. 

For the support property, 
let $C_H \geq 0$ be as in Lemma \ref{lem:support}. 
We can show that the quadratic form 
\[
\widetilde{\Delta}_{[S]}+C_H(H\ch_1^B)^2
\]
satisfies the required property, 
see \cite[Theorem 3.5]{bms16} for the details. 
\end{proof}



\begin{thebibliography}{10}

\bibitem[AB13]{ab13}
D.~Arcara and A.~Bertram.
\newblock Bridgeland-stable moduli spaces for {$K$}-trivial surfaces.
\newblock {\em J. Eur. Math. Soc. (JEMS)}, 15(1):1--38, 2013.
\newblock With an appendix by Max Lieblich.

\bibitem[ABCH13]{abch13}
D.~Arcara, A.~Bertram, I.~Coskun, and J.~Huizenga.
\newblock The minimal model program for the {H}ilbert scheme of points on
  {$\Bbb{P}^2$} and {B}ridgeland stability.
\newblock {\em Adv. Math.}, 235:580--626, 2013.

\bibitem[Bay18]{bay18}
A.~Bayer.
\newblock Wall-crossing implies {B}rill-{N}oether: applications of stability
  conditions on surfaces.
\newblock In {\em Algebraic geometry: {S}alt {L}ake {C}ity 2015}, volume~97 of
  {\em Proc. Sympos. Pure Math.}, pages 3--27. Amer. Math. Soc., Providence,
  RI, 2018.

\bibitem[BLMS17]{blms17}
A.~Bayer, M.~Lahoz, E.~Macrì, and P.~Stellari.
\newblock Stability conditions on {K}uznetsov components, 2017.

\bibitem[BL17]{bl17}
A.~Bayer and C.~Li.
\newblock Brill-{N}oether theory for curves on generic abelian surfaces.
\newblock {\em Pure Appl. Math. Q.}, 13(1):49--76, 2017.

\bibitem[BMS16]{bms16}
A.~Bayer, E.~Macr\`\i, and P.~Stellari.
\newblock The space of stability conditions on abelian threefolds, and on some
  {C}alabi-{Y}au threefolds.
\newblock {\em Invent. Math.}, 206(3):869--933, 2016.

\bibitem[BMT14]{bmt14a}
A.~Bayer, E.~Macr{\`{\i}}, and Y.~Toda.
\newblock Bridgeland stability conditions on threefolds {I}:
  {B}ogomolov-{G}ieseker type inequalities.
\newblock {\em J. Algebraic Geom.}, 23(1):117--163, 2014.

\bibitem[BMW14]{bmw14}
A.~Bertram, C.~Martinez, and J.~Wang.
\newblock The birational geometry of moduli spaces of sheaves on the projective
  plane.
\newblock {\em Geom. Dedicata}, 173:37--64, 2014.

\bibitem[Bir16]{bir16}
C.~Birkar.
\newblock Existence of flips and minimal models for 3-folds in char {$p$}.
\newblock {\em Ann. Sci. \'{E}c. Norm. Sup\'{e}r. (4)}, 49(1):169--212, 2016. 

\bibitem[BW17]{bw17}
C.~Birkar and J.~Waldron.
\newblock Existence of {M}ori fibre spaces for 3-folds in {${\rm char}\,p$}.
\newblock {\em Adv. Math.}, 313:62--101, 2017.

\bibitem[Bog78]{bog78}
F.~A. Bogomolov.
\newblock Holomorphic tensors and vector bundles on projective manifolds.
\newblock {\em Izv. Akad. Nauk SSSR Ser. Mat.}, 42(6):1227--1287, 1439, 1978.

\bibitem[Bri07]{bri07}
T.~Bridgeland.
\newblock Stability conditions on triangulated categories.
\newblock {\em Ann. of Math. (2)}, 166(2):317--345, 2007.

\bibitem[Bri08]{bri08}
T.~Bridgeland.
\newblock Stability conditions on {$K3$} surfaces.
\newblock {\em Duke Math. J.}, 141(2):241--291, 2008.

\bibitem[CT19]{ct19}
P.~Cascini and H.~Tanaka.
\newblock Purely log terminal threefolds with non-normal centres in
  characteristic two.
\newblock {\em Amer. J. Math.}, 141(4):941--979, 2019.

\bibitem[CH14]{ch14}
I.~Coskun and J.~Huizenga.
\newblock Interpolation, {B}ridgeland stability and monomial schemes in the
  plane.
\newblock {\em J. Math. Pures Appl. (9)}, 102(5):930--971, 2014.

\bibitem[CH15]{ch15}
I.~Coskun and J.~Huizenga.
\newblock The birational geometry of the moduli spaces of sheaves on {$\Bbb
  P^2$}.
\newblock In {\em Proceedings of the {G}\"okova {G}eometry-{T}opology
  {C}onference 2014}, pages 114--155. G\"okova Geometry/Topology Conference
  (GGT), G\"okova, 2015.

\bibitem[CH16]{ch16}
I.~Coskun and J.~Huizenga.
\newblock The ample cone of moduli spaces of sheaves on the plane.
\newblock {\em Algebr. Geom.}, 3(1):106--136, 2016.

\bibitem[CHW17]{chw17}
I.~Coskun, J.~Huizenga, and M.~Woolf.
\newblock The effective cone of the moduli space of sheaves on the plane.
\newblock {\em J. Eur. Math. Soc. (JEMS)}, 19(5):1421--1467, 2017.

\bibitem[FL18]{fl18}
S.~Feyzbakhsh and C.~Li.
\newblock Higher rank Clifford indices of curves on a K3 surface, 2018.

\bibitem[Gie79]{gie79}
D.~Gieseker.
\newblock On a theorem of {B}ogomolov on {C}hern classes of stable bundles.
\newblock {\em Amer. J. Math.}, 101(1):77--85, 1979.

\bibitem[GNT19]{gnt19}
Y.~Gongyo, Y.~Nakamura, and H.~Tanaka.
\newblock Rational points on log {F}ano threefolds over a finite field.
\newblock {\em J. Eur. Math. Soc. (JEMS)}, 21(12):3759--3795, 2019.

\bibitem[HW19a]{hw19}
C.~Hacon and J.~Witaszek.
\newblock The minimal model program for threefolds in characteristic five,
  2019.
  
\bibitem[HW19b]{hw20}
C.~Hacon and J.~Witaszek.
\newblock On the relative Minimal Model Program for threefolds in low
  characteristics, 2019.

\bibitem[HX15]{hx15}
C.~D. Hacon and C.~Xu.
\newblock On the three dimensional minimal model program in positive
  characteristic.
\newblock {\em J. Amer. Math. Soc.}, 28(3):711--744, 2015.

\bibitem[HRS96]{hrs96}
D.~Happel, I.~Reiten, and S.~O. Smal{\o}.
\newblock Tilting in abelian categories and quasitilted algebras.
\newblock {\em Mem. Amer. Math. Soc.}, 120(575):viii+ 88, 1996.

\bibitem[Kol13]{kol13}
J.~Koll\'{a}r.
\newblock {\em Singularities of the minimal model program}, volume 200 of {\em
  Cambridge Tracts in Mathematics}.
\newblock Cambridge University Press, Cambridge, 2013.
\newblock With a collaboration of S\'{a}ndor Kov\'{a}cs.

\bibitem[KM98]{km98}
J.~Koll\'{a}r and S.~Mori.
\newblock {\em Birational geometry of algebraic varieties}, volume 134 of {\em
  Cambridge Tracts in Mathematics}.
\newblock Cambridge University Press, Cambridge, 1998.
\newblock With the collaboration of C. H. Clemens and A. Corti, Translated from
  the 1998 Japanese original.


\bibitem[Kos20a]{kos20}
N.~Koseki.
\newblock Stability conditions on Calabi-Yau double/triple solids, 2020.

\bibitem[Kos20b]{kos20b}
N.~Koseki. 
\newblock On the Bogomolov-Gieseker inequality for hypersurfaces in the projective spaces, 2020. 

\bibitem[Lan04]{lan04}
A.~Langer.
\newblock Semistable sheaves in positive characteristic.
\newblock {\em Ann. of Math. (2)}, 159(1):251--276, 2004.

\bibitem[Lan15]{lan15}
A.~Langer.
\newblock Bogomolov's inequality for {H}iggs sheaves in positive
  characteristic.
\newblock {\em Invent. Math.}, 199(3):889--920, 2015.

\bibitem[Lan16]{lan16}
A.~Langer.
\newblock The {B}ogomolov-{M}iyaoka-{Y}au inequality for logarithmic surfaces
  in positive characteristic.
\newblock {\em Duke Math. J.}, 165(14):2737--2769, 2016.

\bibitem[Li19]{li19b}
C.~Li.
\newblock On stability conditions for the quintic threefold.
\newblock {\em Invent. Math.}, 218(1):301--340, 2019.

\bibitem[LZ16]{lz16}
C.~{Li} and X.~{Zhao}.
\newblock {Birational models of moduli spaces of coherent sheaves on the
  projective plane}.
\newblock {\em ArXiv e-prints}, March 2016.

\bibitem[LZ18]{lz18}
C.~Li and X.~Zhao.
\newblock The minimal model program for deformations of {H}ilbert schemes of
  points on the projective plane.
\newblock {\em Algebr. Geom.}, 5(3):328--358, 2018.

\bibitem[MS17]{ms17}
E.~Macr\`i and B.~Schmidt.
\newblock Lectures on {B}ridgeland stability.
\newblock In {\em Moduli of curves}, volume~21 of {\em Lect. Notes Unione Mat.
  Ital.}, pages 139--211. Springer, Cham, 2017.

\bibitem[Muk13]{muk13}
S.~Mukai.
\newblock Counterexamples to {K}odaira's vanishing and {Y}au's inequality in
  positive characteristics.
\newblock {\em Kyoto J. Math.}, 53(2):515--532, 2013.

\bibitem[Ray78]{ray78}
M.~Raynaud.
\newblock Contre-exemple au ``vanishing theorem'' en caract\'{e}ristique
  {$p>0$}.
\newblock In {\em C. {P}. {R}amanujam---a tribute}, volume~8 of {\em Tata Inst.
  Fund. Res. Studies in Math.}, pages 273--278. Springer, Berlin-New York,
  1978.
  
\bibitem[SB91]{sb91}
N.~I. Shepherd-Barron.
\newblock Unstable vector bundles and linear systems on surfaces in
  characteristic {$p$}.
\newblock {\em Invent. Math.}, 106(2):243--262, 1991. 

\bibitem[Sun19]{sun19}
Hao Sun.
\newblock Bogomolov's inequality for product type varieties in positive
  characteristic, 2019.

\bibitem[Tan20]{tan20}
H.~Tanaka.
\newblock Abundance theorem for surfaces over imperfect fields.
\newblock {\em Math. Z.}, 295(1-2):595--622, 2020.

\end{thebibliography}
\end{document}